\documentclass[11pt,reqno]{amsart}

\usepackage[utf8]{inputenc}

\usepackage{graphicx,amssymb,mathrsfs,amsmath,amsfonts,appendix, subfigure, amsthm, paralist, enumitem, comment, mathtools, esint, derivative, colonequals}
\usepackage[colorlinks=true, pdfstartview=FitV, linkcolor=blue, citecolor=blue, urlcolor=blue]{hyperref}
\usepackage{comment}
\usepackage{orcidlink}

\numberwithin{equation}{section}
\newtheorem{lemma}{Lemma}[section]
\newtheorem{theorem}{Theorem}[section]

\newtheorem{proposition}{Proposition}[section]

\newtheorem{corollary}{Corollary}[section]

\newcommand{\R}{\mathbb{R}}

\title[Propagation of chaoticity for the Motsch--Tadmor model]{On the uniform-in-time propagation of chaos for the Motsch--Tadmor model}

\author{T.~Kldiashvili\,}
\address{Institute of Mathematics and Scientific Computing, University of Graz, Heinrichstraße 36, 8010 Graz (Austria); De Vinci Higher Education, De Vinci Research Center, Paris (France)}
\email{tamari.kldiashvili@uni-graz.at}

\author{S.~Rossi\,\orcidlink{0000-0003-4305-3404}}
\address{Sapienza Università di Roma, Dipartimento di Matematica, Piazzale Aldo Moro 5, 00185  Rome (Italy)}
\email{stefano.rossi2@uniroma1.it}

\author{L.~Trussardi\,\orcidlink{0000-0001-6255-9983}}
\address{Institute of Mathematics and Scientific Computing, University of Graz, Heinrichstraße 36, 8010 Graz (Austria)}
\email{lara.trussardi@uni-graz.at}
\date{\today}

\begin{document}

\begin{abstract}
    We consider the Motsch--Tadmor model, a modification of the Cucker–Smale model for flocking dynamics that incorporates asymmetric communication weights. In the case of uniformly positive interactions, we consider the validity problem for the limiting kinetic equation, proving a uniform-in-time propagation of chaos result in the moving frame given by the consensus velocities. The result is obtained by studying the reversibility of the communication weights and by implicitly characterizing the asymptotic consensus. 
\end{abstract}
\maketitle

\section{Introduction}
In this article, we study the \emph{Motsch--Tadmor model}, an interacting particle system describing collective alignment and flocking behavior in multi-agent systems. The system reads as
\begin{equation}\label{eq:MTmodel}
\begin{cases}
\displaystyle
\dot X_i^N(t)=V_i^N(t),\\[8pt]
\displaystyle
\dot V_i^N(t)
=
\sum_{j=1}^N
\frac{
\varphi\bigl(|X_i^N(t)-X_j^N(t)|\bigr)
}{
\sum_{\ell=1}^N
\varphi\bigl(|X_i^N(t)-X_\ell^N(t)|\bigr)
}
\bigl(V_j^N(t)-V_i^N(t)\bigr),
\end{cases}
\quad i=1,\ldots,N
\end{equation}
where \(X_i^N(t),V_i^N(t)\in \R^d\) respectively denote the position and velocity of the \(i\)-th agent at time \(t\), and \(\varphi: \R_+ \to\R_+\) is a suitable regular \emph{influence function} characterizing the strength of the interaction between two agents as a function of their mutual distance.

This model was introduced by Motsch and Tadmor in~\cite{MT11} 
as a modification of the mean-field Cucker--Smale model (see~\cite{HT08,CS1,CS2}) 
to overcome the ``interaction-intensity problem''.
Namely, the fact that,
in the Cucker--Smale model in the mean-field scaling, 
each agent is influenced by every other agent of the population with an intensity of \(O(1/N)\) that depends on the size of the system. 
In~\cite{MT11}, to take into account the effect of nearby agents on the dynamics,
the authors  replaced the interaction weight $\varphi(|X_i-X_j|)/N$ between agents \(X_i\) and \(X_j\) with the so-called \emph{relative influence}, defined by 
\begin{equation}
\label{eq:relinf}
a_{ij}(\boldsymbol{X}^N):= \dfrac{\varphi(|X_i^N-X_j^N|)}{\sum_{\ell=1}^N \varphi(|X_i^N-X_\ell^N|)}, \qquad \boldsymbol{X}^N:=(X_1^N,\dots,X^N_N).
\end{equation}

As it is important to notice, the \emph{communication weights} \(a_{ij}\) are generally not symmetric. 
It follows that, unlike the Cucker--Smale model, the Motsch--Tadmor model does not, in general, conserve total momentum. 
As a consequence, the asymptotic consensus velocity
is not generally determined by the initial data through momentum conservation, but rather emerges from the collective dynamics of the system.
Moreover, the standard \(\ell^2\) energy estimates available for Cucker--Smale-type systems do not directly provide a dissipative estimate for the relative velocities, making the analysis of convergence towards consensus more challenging.

To overcome these difficulties, other approaches have been developed to investigate the long-time behavior for \eqref{eq:MTmodel}. In their original work \cite{MT11}, Motsch and Tadmor established exponential decay of the velocity diameter under suitable assumptions on the interaction, thereby proving convergence of the agents' velocities towards a common asymptotic value.

This strategy was later extended to the kinetic setting by Tan~\cite[Theorem 2.3 and Proposition 2.4]{Tan2014}, who obtained flocking estimates and existence of global regular solutions for the  kinetic version of the Motsch--Tadmor model under compactly-supported regular initial data $f_{\mathrm{in}}\in C^1_{\mathrm c}$. 
Further results can be found in~\cite{gao2017global, ha2026emergent}.
More recent developments concerning mono-cluster and multi-cluster flocking configurations, including time-discrete formulations of the  Motsch--Tadmor model, can be found in~\cite{HJZ25,Ha2026, HJLZ26}. We also refer to~\cite{HaLiu09,HT08,NP22,HaWang26} for related results on the Cucker--Smale model.

In this work, 
we exploit the detailed-balance structure 
of the communication weights showing that,
under suitable assumptions on the influence function,
the asymmetric communication matrix \(\{a_{ij}(\boldsymbol{X}^N)\}_{i,j=1}^N\) 
is reversible at each fixed configuration \(\boldsymbol{X}^N\) with respect to a suitable time-dependent distribution \(\{\pi^N_i(t)\}_{i=1}^N\). 

This point of view provides an implicit characterization of the consensus velocity \(\mathbf V_\infty^N\) for the discrete system, together with an analogous quantity at the kinetic level.
In Proposition \ref{thm1}, we characterize the consensus velocity as the limit as \(t \uparrow \infty\) of the weighted average of the velocities taken respect to the reversible distribution of the communication matrix.

The asymptotic consensus in non-symmetric systems has also already been investigated in the past. Boudin, Salvarani, and Trélat in~\cite{BoudinSalvaraniTrelat2022} considered linear first-order consensus dynamics with \emph{time-independent} interaction coefficients and established the existence of an invariant distribution, under suitable connectivity assumptions, and showed that the corresponding weighted average is conserved along the dynamics, allowing to explicitly identify the asymptotic consensus in terms of the initial data. Related results for flocking dynamics on weighted graphs were obtained by Cotil~\cite{Cotil23} through a probabilistic approach.
Concerning models with time-dependent coefficients, see e.g. \cite{BonnetPouradierDuteilSigalotti2022, BonnetWeillSigalotti2024}. 

Our main contribution uses the same detailed-balance structure to provide the correct \(\ell^2\) weighted norm for comparing the particle system and the kinetic equation corresponding to \eqref{eq:MTmodel}, thus establishing uniform-in-time propagation of chaos result. 
Recenlty, various works have been devoted to prove propagation of chaos results for consensus dynamics so that the kinetic limit holds uniformly-in-time, see e.g. \cite{GerberVaes2026,Ha2026, HaHoffmannKim2026}.
However, to the author's knowledge the uniform-in-time derivation for the Motsch--Tadmor is still an open question (see \cite{Ha2025, Tadmor2023, BBC20}). We address this issue in Theorem \ref{thm2} of the present work.

Of course other particle systems share the asymmetric features of the Motsch--Tadmor model and would merit a thorough study in the future. We mention here models with topological interactions, dynamical formulation for transformers, and more general non-exchangeable interacting systems, see e.g. \cite{DP19, BenedettoCagliotiRossi2021, DegondPulvirentiRossi2023, BenedettoPaulRossi2025, MenciPaulRossiTenna2026, DuerinckxGeshkovskiRossi2026, paul2026meanfield,PaulRossiTrelat2025}.

The structure of the paper is as follows: In Section \ref{sec2}, we introduce the main technical assumptions and present the two main theorems of this article. In Section \ref{sec3}, we characterize the asymptotic consensus velocity, while in Section \ref{sec4}, we prove the uniform-in-time propagation of chaos. Appendix \ref{sec:appendix} contains the proof of the consistency error estimate needed in the proof of the uniform-in-time propagation of chaos.

\section{Setting and main results}
\label{sec2}
We consider the Motsch-Tadmor model introduced in \eqref{eq:MTmodel}-\eqref{eq:relinf}.
Moreover, we assume that the influence function \(\varphi\) is such that
\(\varphi\in W^{1,\infty}(\R_+)\)
and that there exist \(0<m\le M\) such that
\begin{equation}
\label{in:lowupbound}
    0<m\leq \varphi(r)\leq M<\infty
\quad\text{for all }r\geq0.
\end{equation}
We again notice that the model \eqref{eq:MTmodel} is no longer symmetric, since in general $a_{ij} \neq a_{ji}$ for $i \neq j$, due to the relative influence denominator. 

Our first result describes the asymptotic velocity consensus reached by the agents of system \eqref{eq:MTmodel} and by the solutions of the related kinetic equation.
Under assumption \eqref{in:lowupbound}, it is known that exponential flocking holds in terms of the velocity diameter
\begin{equation}
\label{def:suppdv}
D^N_V(t):= \max_{i,j \in \{1,\dots, N\}}\left|V^N_i(t)-V_j^N(t) \right|.
\end{equation}
Indeed, since the communication weights satisfy
\[
N\, a_{ij}(\boldsymbol{X}^N)\geq \frac{m}{M},
\]
the diameter contraction estimate in \cite[Theorem 2.3]{MT14} yields
\begin{equation}
\label{in:asymptfl}
D^N_V(t)\le D^N_V(0) {\mathrm e}^{- \eta t}, \quad \eta:=\frac mM.
\end{equation}
This implies that the velocity diameter of the agents is contracting exponentially over time, thus demonstrating the existence of an asymptotic consensus velocity $\mathbf{V}^N_\infty$. Unlike in symmetric models, its value is not determined by the total momentum, which is not conserved in this case. 

A similar exponential flocking estimate holds at the kinetic level. 
Let us consider the empirical measure of a system with $N$ agents, given by
\begin{align}
    \mu^N_t(x,v) := \frac{1}{N}\sum\limits_{i=1}^N\delta_{X_i^N(t)}(x)\otimes\delta_{V^N_i(t)}(v)\in \mathcal{P}(\R^{d}\times \R^d),
\end{align}
where $\delta_x\otimes\delta_v$ is the Dirac delta centered in $(x,v) \in \R^d\times\R^d$. It is not difficult to show that the empirical measure \(\mu^N_t\) weakly solves the {non-linear} Motsch--Tadmor kinetic equation 
\begin{equation}
\label{eq:mesoVlasov}
\partial_t f_t(x,v)
+
v\cdot\nabla_x f_t(x,v)
+
\nabla_v\cdot\bigl(F[f_t]\,f_t(x,v)\bigr)
=0,
\end{equation}
where, for any \(t \geq 0\), we denote by \(f_t(x,v)\in \mathcal{P}(\R^{d}\times \R^d)\) the time-dependent probability distribution of agents in the phase-space \((x,v) \in \R^d \times \R^d\) and
\begin{equation}
\label{eq:meanforce}
F[f_t](x,v)
:=
\displaystyle
\int_{\R^{d}\times \R^d}
a^{\rho}(t,x,y)\, (w-v)\, {\rm d}f_t(y,w)
, \quad a^{\rho}(t,x,y):= \dfrac{\varphi(|x-y|) }{\displaystyle \int_{\R^{d}} \varphi(|x-\bar y|) \mathrm d\rho_t(\bar y)},
\end{equation}
where \(\rho_t(x):= \int_{\R^d_v} f_t(x,\mathrm dv) \).
Notice that the lower bound in~\eqref{in:lowupbound} implies a uniform non-degeneracy of the denominators in~\eqref{eq:meanforce}. 

We mention that existence results for solutions to the kinetic equation~\eqref{eq:mesoVlasov} are present in the literature, see e.g. Tan~\cite{Tan2014} for regular solutions and Karper, Mellet and Trivisa~\cite{KMT13} for weak solutions. 
In the following, we will work with compactly supported regular solutions \(f_t \in C^1_{\mathrm c}(\R^d \times \R^d)\) to \eqref{eq:mesoVlasov}.

Introducing the velocity diameter
\begin{equation}
\label{def:suppdvkin}
D^f_V(t):= \sup_{(x,v),(y,w)\in \mathrm{supp} f_t}\left|v-w \right|
\end{equation}
and following the contraction argument in \cite[Proposition 2.1]{Tan2014}, 
it can be shown that exponential flocking holds in terms of \(D_V^f\).  
Indeed, although the result in \cite{Tan2014} is stated for non-increasing functions \(\varphi\),
the same argument applies
under assumption \eqref{in:lowupbound}, since
\[
a^\rho(t,x,y) \geq \frac{m}{M}.
\]
This implies that
\begin{equation}
\label{in:asymptflkin}
D^f_V(t)\le D^f_V(0) {\mathrm e}^{- \eta t}, \quad \eta:=\frac mM,
\end{equation}
thus demonstrating the existence of an asymptotic consensus velocity $\mathbf{v}^f_\infty$.

By introducing the distributions 
\begin{equation}
\label{def:eqmeas}
\pi^N_i(t):=\dfrac{\sum_{\ell}\varphi(|X_i^N(t)-X_\ell^N(t)|)}{\sum_{\ell,j}\varphi(|X^N_\ell(t)-X^N_j(t)|)}, \quad 
\pi^\rho(t,x):= \dfrac{\int\varphi(|x-y|)\mathrm{d}{\rho}_t(y)}{\int \varphi(|\bar x-\bar y|)\mathrm{d}\rho_t(\bar y) \mathrm d\rho_t(\bar x)},
\end{equation}
we prove that (see Section \ref{sec3}) the weights \(\{a_{ij}(\boldsymbol{X}^N(t))\}\) and \(a^{\rho}(t,x,y)\) are reversible w.r.t. these expressions. This implies the following characterizations of the limit \(\mathbf{V}^N_\infty\) and \(\mathbf{v}^f_\infty\).
\begin{proposition}[Velocity consensus]
\label{thm1}
    Consider the Motsch--Tadmor model \eqref{eq:MTmodel} with initial data \(\{(X_i^N(0), V_i^N(0))\}_{i=1}^N\) and assume that the influence function \(\varphi\) satisfies the assumption in \eqref{in:lowupbound}. Introducing the quantity
    \begin{equation}
    \label{eq:weightedN}
    \langle \boldsymbol{V}^N(t)\rangle_{N}:= \sum_{i=1}^N \pi_i^N(t) V_i^N(t),
    \end{equation}
    where \(\{\pi_i^N(t)\}_{i=1}^N\) is given in \eqref{def:eqmeas}, there exists \(\mathbf{V}^N_\infty\in \R^d\), characterized in \eqref{eq:vinf}, such that \( \langle \boldsymbol{V}^N(t)\rangle _{N} \to \mathbf{V}^N_\infty\) as \(t \uparrow \infty\). Moreover, there exists \(C>0\) such that
    \[
    \max_{i\in \{1,\dots,N\}}\left|V_i^N(t)-\mathbf{V}^N_\infty\right|\le C \, \mathrm e^{-\eta t}, \quad \eta:= \frac mM.
    \]
    Similarly, consider a regular solution \(f_t(x,v)\) of the Motsch--Tadmor model \eqref{eq:mesoVlasov} with compactly supported initial datum \(f_{\mathrm{in}} \in C^1_{\mathrm c}(\R^d \times \R^d)\)
    and introduce the quantity
    \begin{equation}
    \label{eq:weightedf}
    \langle v(t)\rangle_{f}:= \displaystyle \int_{\R^d\times \R^d}\pi^\rho(t,x) v\,  \mathrm df_t(x,v),
    \end{equation}
    where \(\pi^\rho(t,x)\) is given in \eqref{def:eqmeas}. Then, there exists \(\mathbf{v}^f_\infty\in \R^d\), characterized in \eqref{eq:vinff}, such that \( \langle v(t)\rangle _{f} \to \mathbf{v}^f_\infty\) as \(t \uparrow \infty\). Moreover, there exists \(C>0\), such that
    \[
    \left|v-\mathbf{v}^f_\infty\right|\le C \, \mathrm e^{-\eta t}, \quad \eta:= \frac mM,
    \]
    for any \(v \in \text{supp}f_t\).
\end{proposition}

The previous result is the key ingredient to prove our propagation of chaos result. Given the kinetic equation~\eqref{eq:mesoVlasov}, we aim to prove the uniform-in-time validity of the kinetic model starting from the particle system. 
For $N \in \mathbb{N}_{>0}$,
let $\Phi_t^N$ denote the particle-flow of system \eqref{eq:MTmodel} and, 
given $F^N_{\text{in}} \in \mathcal{P}(\R^{2dN})$, 
consider the push-forward measure \(F^N_t:=(\Phi^N_t)_\# F^N_{\text{in}}\), weak solution to the \emph{linear Liouville} equation 
\begin{equation}
    \label{linearLiouville}
\begin{aligned}
    &\partial_t F^N_t(X_N,V_N) + \sum_{i=1}^N v_i \cdot\nabla_{x_i}F^N_t(X_N,V_N)\\
    &+ \frac1 N\sum_{i=1}^N\nabla_{v_i} \cdot \left(\sum_{j=1}^N \frac{\varphi(|x_i-x_j|)}{\frac{1}{N}\sum_{h=1}^N\varphi(|x_h-x_i|)}(v_j-v_i) F^N_t(X_N,V_N) \right)=0,
\end{aligned}
\end{equation}
where \(X_N:=(x_1,\dots,x_N)\) and \(V_N:=(v_1,\dots,v_N)\).
We then define the marginal distributions 
\begin{equation}
\label{def:kmarg}
F^{N}_{k,t}(X_k, V_k):= \int\limits_{\R^{2d(N-k)}}F^N_t(X_N,V_N)\,dx_{k+1}\,dv_{k+1}\,\dots\,dx_N\,dv_N, \quad 1\le k \le N-1
\end{equation}
and we assume that chaos holds at time zero, i.e. the initial data for \eqref{linearLiouville} is of the form $F^N_{\text{in}}:=f_{\text{in}}^{\otimes N}$ where $f_{\text{in}}$ is the initial datum of the kinetic equation~\eqref{eq:mesoVlasov}. 

For a fixed group of \(k<N\) particles, we expect that chaos is propagated at positive times $t$ as \(N \uparrow \infty\), i.e. weakly $\lim_{N \uparrow \infty} F^{N}_{k,t}=f^{\otimes k}_t$ where $f_t$ is solution to \eqref{eq:mesoVlasov} with initial datum $f_{\text{in}}$. 
However,
a uniform-in-time estimate for the absolute positions cannot be expected
without removing the collective motion of the flock. 
For this reason, we obtain our uniform-in-time propagation of chaos result by passing to the moving frame determined by the consensus velocities of the two systems. Namely, we will compare
\begin{equation}
\label{eq:movingframe1}
\bigl(
X_i^N(t)-t\mathbf V_\infty^N,\,
V_i^N(t)-\mathbf V_\infty^N
\bigr)
\quad 
\text{and}
\quad
\bigl(
\bar X_i(t)-t\mathbf v^f_\infty,\,
\bar V_i(t)-\mathbf v^f_\infty
\bigr),
\end{equation}
where \(\{X_i^N(t), V_i^N(t)\}_{i=1}^N\) are given in \eqref{eq:MTmodel} and \(\{\bar X_i(t), \bar V_i(t)\}_{i=1}^N\) are the positions and velocities of \(N\) independent particles evolving through the mean-field interaction term \eqref{eq:meanforce}; see also \eqref{eq:interdyn}.
The initial data for \(\{X_i^N(t), V_i^N(t)\}_{i=1}^N\) and \(\{\bar X_i(t), \bar V_i(t)\}_{i=1}^N\) in \eqref{eq:movingframe1} are \((X_N, V_N) \in \mathrm{supp} (f_{\mathrm{in}}^{\otimes N})\).

 To quantify the discrepancy between two probability measures $\mu, \nu$ on $\R^m$  with finite second moments, we use the standard \(2\)-Wasserstein distance
  \begin{align*}
        \mathcal{W}_2(\mu,\nu)^2 = \inf\limits_{\gamma \in \Gamma(\mu,\nu)}\int\limits_{\R^m\times\R^m} |x-y|^2 \, {\mathrm d\gamma}(x,y),
    \end{align*}
where $\Gamma(\mu,\nu)$ denote the set of couplings between $\mu$ and $\nu$, i.e. the set of probability densities on $\R^m\times\R^m$ whose marginals are $\mu$ and $\nu$.

In the following theorem, we prove that the interacting particle system is uniformly chaotic in time with respect to the moving frame compatible with the consensus velocities, with mean-field rate $N^{-1/2}$ for every fixed $k$.

\begin{theorem}[Uniform-in-time propagation of chaos]
\label{thm2}
Assume \eqref{in:lowupbound}, and suppose \(f_{\mathrm{in}}\in C_{\mathrm c}^{1}(\R^{d}\times \R^d)\). 
Let \(f_t\) be the regular solution of the kinetic Motsch--Tadmor model \eqref{eq:mesoVlasov} with initial datum \(f_{\mathrm{in}}\) and let \(F_{t}^N\) be the Liouville solution to \eqref{linearLiouville} with initial datum \(F^N_{\mathrm{in}}=f_{\mathrm{in}}^{\otimes N}\).

Then there exists a constant \(C>0\), depending only on \(d\), \(m\), \(M\), \(\|\varphi\|_{\mathrm{Lip}}\), and the diameter of the velocity support of \(f_{\mathrm{in}}\), such that
\begin{equation}
\label{in:propchaos}
\sup_{t\geq0}\,
\mathcal{W}_2\left(
\widetilde F_{k,t}^N,
\widetilde f_t^{\otimes k}
\right)
\leq
C\sqrt{\frac{k}{N}},
\end{equation}
where
\(\widetilde f_t:= (\mathcal{G}_t)_\# f_t\)
and, for \(1\leq k\leq N\), 
\(\widetilde F_{k,t}^N\) is the \(k\)-particle marginal of 
\((\mathcal{G}_t^N)_{\#}F^N_{t}\) where
$
\mathcal G_t(x,v)
:=
\left(
x-t\mathbf v^f_\infty,\,
v-\mathbf v^f_\infty
\right)
$ and
\begin{equation}
    \label{eq:movingframe}
    \mathcal{G}^N_{t}(X_N,V_N):=(x_1-t \mathbf V_\infty^N,\dots, x_N -t \mathbf V_\infty^N, v_1- \mathbf V^N_\infty, \dots, v_N-\mathbf V_\infty^N).
\end{equation}
\end{theorem}

Two corollaries easily follow by Proposition \ref{thm1} and Theorem \ref{thm2}. The first characterizes asymptotic flocking at the Eulerian level of the solutions of the Liouville equation \eqref{linearLiouville} and the mesoscopic equation \eqref{eq:mesoVlasov}. In particular, along the moving frame introduced in Theorem \ref{thm2}, the Liouville and Vlasov solutions converge to tensorized measures with Dirac deltas in the velocities.

\begin{corollary}[Asymptotic flocking in the moving frame]
\label{cor1}
Assume \eqref{in:lowupbound}, and let \(f_{\mathrm in}\in \mathcal C^1_{\mathrm c}(\R^d\times\R^d)\) be compactly supported. 
Let \(f_t\) be the solution to \eqref{eq:mesoVlasov}, and let
\(
\widetilde f_t=(\mathcal G_t)_\# f_t.
\)
Then there exist \(\rho_\infty\in\mathcal P(\R^d)\) and a constant \(C>0\) such that
\begin{equation}
\label{eq:kinetic-asy}
\mathcal W_2\left(
\widetilde f_t,\,
\rho_\infty\otimes\delta_0
\right)
\leq
C\,\mathrm e^{-\eta t}.
\qquad
\eta=\frac mM,
\end{equation}
Similarly, let
\(
F_t^N=(\Phi_t^N)_\#F_{\rm in}^N
\)
be a solution of the Liouville equation \eqref{linearLiouville}, and let
\(
\widetilde F_{k,t}^N
\) be the \(k\)-th marginal of 
\((\mathcal G_t^N)_\#F_t^N
\)
where
\(
\mathcal G_t^N(X_N,V_N)
\)
is defined in \eqref{eq:movingframe}.
Then there exist
\(
\rho_{k,\infty}^N
\in\mathcal P(\R^{dk})
\)
and a constant \(C>0\) such that
\begin{equation}
\label{eq:particle-asy}
\mathcal W_2\left(
\widetilde F_{k,t}^N,\,
\rho_{k,\infty}^N\otimes\delta_0^{\otimes k}
\right)
\leq
C\,\sqrt k\,\mathrm e^{-\eta t}.
\end{equation}
\end{corollary}

\begin{proof}
We give a proof for the discrete case.
By Proposition~\ref{thm1}, for any \(i=1,\ldots,N\) there exists a constant \(C>0\) such that
\[
\left|
V_i^N(t)-\mathbf V_\infty^N
\right|
\leq
C\,\mathrm e^{-\eta t}.
\]
Since
\[
\frac{\mathrm d}{\mathrm dt}
\left(
X_i^N(t)-t\mathbf V_\infty^N
\right)
=
V_i^N(t)-\mathbf V_\infty^N,
\]
we deduce that
\[
X_{i,\infty}^N
:=
X_i^N(0)
+
\int_0^\infty
\left(
V_i^N(s)-\mathbf V_\infty^N
\right)\,\mathrm ds
\]
exists and
\[
\left|
X_i^N(t)
-t\mathbf V_\infty^N
-X_{i,\infty}^N
\right|
\leq
\frac C \eta\,\mathrm e^{-\eta t}.
\]
This yields to 
\[
\left|
\mathcal G^N_t\bigl(\boldsymbol{X}^N(t),\boldsymbol{V}^N(t)\bigr)
-
\bigl(X^N_\infty,0\bigr)
\right|^2
\leq
C^2\, N \left(1 +\frac{1}{\eta^2}\right)\mathrm e^{-2\eta t},
\]
where \(
X_{\infty}^N
:=
\bigl(
X_{1,\infty}^N,\ldots,X_{N,\infty}^N
\bigr).
\)
Setting \(
 \rho_\infty^N:=(X_{\infty}^N)_\#F_{\rm in}^N\), 
the definition of Wasserstein distance implies that
\[
\mathcal W_2^2 \left(
\widetilde F_t^N,\,
\rho_\infty^N\otimes\delta_0^{\otimes N}
\right)
\leq 
\int_{(\R^d \times \R^d)^N} \left|
\mathcal G^N_t\bigl(\boldsymbol{X}^N(t),\boldsymbol{V}^N(t)\bigr)
-
\bigl(X^N_\infty,0\bigr)
\right|^2\mathrm d f_{\mathrm{in}}^{\otimes N}
\leq
C\,N\, \mathrm e^{-2\eta t},
\]
which concludes the proof.
\end{proof}

The second corollary shows that the two limits as \(N \uparrow \infty\) and \(t \uparrow \infty\) commute. The proof follows by the triangle inequality and the uniform-in-time result of Theorem \ref{thm2} and it is omitted.
\begin{corollary}
Under the assumptions and notation of Theorem \ref{thm2},
there exists a constant \(C>0\), depending only on \(d\), \(\eta\), \(\mathrm{Lip}(\varphi)\), and the diameter of the velocity support of \(f_{\mathrm{in}}\), such that
\begin{equation}
\mathcal{W}_2\left(
\widetilde F_{k,t}^N,
 (\rho_\infty\otimes \delta_0)^{\otimes k}
\right)
\leq
C\sqrt{k}\left({\frac{1}{\sqrt{N}}}+  \mathrm e^{-\eta t}\right),
\end{equation}
    for fixed \(k \leq N-1\).
\end{corollary}

\section{Consensus in the Motsch--Tadmor model: proof of Proposition \ref{thm1}}
\label{sec3}
We start by proving the following lemma concerning the distributions \(\{\pi^N_i\}_{i=1}^N\) and \(\pi^\rho(t,x)\) introduced in \eqref{def:eqmeas}.
\begin{lemma}
\label{lemma:lemmapi}
Consider the Motsch--Tadmor dynamics in \eqref{eq:MTmodel} and \eqref{eq:mesoVlasov} and let \(\{\pi_i^N(t)\}_{i=1}^N\) and \(\pi^\rho(t,x)\) as in \eqref{def:eqmeas}. 
Under the assumption \eqref{in:lowupbound}, the following results hold:
\begin{enumerate}
    \item the weights \(\{a_{ij}^N\}\) defined in \eqref{eq:relinf} are reversible w.r.t. \(\{\pi_i^N(t)\}_{i=1}^N\), i.e. \(\sum_{i=1}^N \pi^N_i(t)=1\) and
    \begin{equation}
    \label{eq:detbal}
    \pi^N_i(t) a_{ij}^N(t)=\pi^N_j(t) a_{ji}^N(t).
    \end{equation}
    Similarly, \(a^{\rho}(t)\) defined in \eqref{eq:meanforce} is reversible w.r.t. \(\pi^\rho(t)\), i.e. \(\int \pi^\rho(t,x)\mathrm d\rho_t(x)=1\) and
    \begin{equation}
    \label{eq:detbal2}
    \pi^\rho(t,x) a^{\rho}(t,x,y)=\pi^\rho(t,y) a^{\rho}(t,y,x).
    \end{equation}
    \item It holds that
    \begin{equation}
\label{in:boundpi}
\max_{1\le i\le N}\left|\dfrac{\dot \pi_i^N(t)}{\pi_i^N(t)}\right|\le \dfrac{2 \text{Lip}(\varphi)}{m}\, D_V^N(t),\quad
\left|\dfrac{\partial_t \pi^\rho+v\cdot \nabla_x \pi^\rho}{\pi^\rho}\right| \le  \frac{2 \text{Lip}(\varphi)}{m} D_V^f(t)
\end{equation}
where
\(D^N_V(t)\) is introduced in \eqref{def:suppdv}, while \(
D^f_V(t)
\) 
in \eqref{def:suppdvkin}.
\item For any \(i \in \{1,\dots,N\}\), it holds that
\[
\eta\le\min \left\{N\pi^N_i(t), \pi^\rho(t,x)\right\} \le \max \left\{N\pi^N_i(t), \pi^\rho(t,x)\right\}\le \dfrac{1}{\eta}, \quad \eta=\dfrac mM.
\]
\end{enumerate}
\end{lemma}
\begin{proof}
We prove the result at the discrete level, a similar argument works for the kinetic quantities.
If \(\varphi_{ij}:=\varphi(|X_i^N-X_j^N|)\), we get
\[
\pi^N_i a^N_{ij}=\dfrac{\sum_\ell \varphi_{i\ell}}{\sum_{\ell h}\varphi_{\ell h}} \dfrac{\varphi_{i j}}{\sum_{\ell} \varphi_{i \ell}}=
\dfrac{\sum_\ell \varphi_{j \ell}}{\sum_{\ell h}\varphi_{\ell h}} \dfrac{\varphi_{j i}}{\sum_{\ell} \varphi_{j \ell}}=
\pi^N_j\, a^N_{j i}\,
\]
since \(\varphi_{ij}=\varphi_{ji}\), this proves the first item. Concerning \eqref{in:boundpi}, we have
\begin{align}
\log \pi_i^N(t)=\log\left(\sum_\ell \varphi(|X_i^N(t)-X_\ell^N(t)|)\right)- \log\left( \sum_{\ell,h} \varphi(|X_\ell^N(t)-X_h^N(t)|\right).
\end{align}
Taking time derivatives, using \eqref{in:lowupbound} and the definition of \(D_V^N(t)\), we get \eqref{in:boundpi}.
Finally, last item follows by the uniform bounds in \eqref{in:lowupbound}. 
\end{proof}
\begin{proof}[Proof of Proposition \ref{thm1}]
We consider the weighted average of the velocities introduced in \eqref{eq:weightedN}
and we show that the limit in time of this quantity exists and corresponds to the consensus velocity of the model. Taking the time derivative
\[
\dfrac{{\mathrm d}}{\mathrm dt} \langle \boldsymbol{V}^N(t)\rangle_{N}= \sum_{i=1}^N   \dot \pi_i^N(t) V_i^N(t)+  \sum_{i=1}^N \pi_i^N(t) \dot V_i^N(t),
\]
by a straightforward computation, the second term is identically equal to zero thanks to \eqref{eq:detbal}.
Hence
\[
\dfrac{{\mathrm d}}{\mathrm dt} \langle \boldsymbol{V}^N(t)\rangle_{N}= \sum_{i=1}^N   \dot \pi_i^N(t) V_i^N(t)= \sum_{i=1}^N   \dot \pi_i^N(t)\left[ V_i^N(t)-  \langle \boldsymbol{V}^N(t)\rangle_{N}\right],
\]
since \(\sum_i \dot \pi_i=0\). By \eqref{in:boundpi}, it follows that
\[
\left| \dfrac{{\mathrm d}}{\mathrm dt} \langle \boldsymbol{V}^N(t)\rangle_{N}\right| \le \sum_{i=1}^N \pi_i^N(t)\left|\dfrac{\dot \pi_i^N(t)}{\pi_i^N(t)}\right|\sum_{j=1}^N \pi_j^N(t)|V_i^N(t)-V_j^N(t)|\le \dfrac{2\text{Lip}(\varphi)}{m}[D^N_V(t)]^2.
\]
By the asymptotic flocking \eqref{in:asymptfl}, we conclude that
\[
\left| \dfrac{{\mathrm d}}{\mathrm dt} \langle \boldsymbol{V}^N(t)\rangle_{N}\right|  \le  \dfrac{2 \text{Lip}(\varphi)}{m} [D^N_V(0)]^2 \mathrm{e}^{-2\eta t}.
\]
This implies that the limit in time of \(\langle \boldsymbol{V}^N(t) \rangle_{N}\) exists and is given by 
\begin{equation}
\label{eq:vinf}
\mathbf{V}^N_\infty:= \langle \boldsymbol{V}^N(0) \rangle_{N}+\int_0^\infty \sum_{i=1}^N \dot \pi_i^N(t) V_i^N(t)\, {\mathrm d}t.
\end{equation}
Moreover
\begin{align*}
|V_i^N(t)- \mathbf{V}_\infty^N| &\le |V_i^N(t)- \langle \boldsymbol{V}^N(t) \rangle_{N}| + |\langle \boldsymbol{V}^N(t) \rangle_{N}- \mathbf{V}_\infty^N|\\
&\le \sum_{j=1}^N \pi_j^N(t) |V_i^N(t)- V_j^N(t)| + |\langle \boldsymbol{V}^N(t) \rangle_{N}- \mathbf{V}_\infty^N|\le C{\mathrm e}^{-\eta t},
\end{align*}
with \(
C:=D^N_V(0)+\text{Lip}(\varphi) [D^N_V(0)]^2/ \eta m.
\)

Similarly, at the kinetic level we can
consider the average velocity introduced in \eqref{eq:weightedf}.
Taking the time derivative and integrating by parts, we get
\begin{align}
\dfrac{{\mathrm d}}{\mathrm dt} \langle v(t)\rangle_{f}&=\int \partial_t \pi^\rho(t,x) v\, \mathrm df_t(x,v)+  \int  \pi^\rho(t,x) v\, \mathrm \partial_tf_t(x,v) \mathrm dx \mathrm dv\nonumber\\
&=\int \left(\partial_t \pi^\rho(t,x) + v \cdot \nabla_x \pi^\rho(t,x)\right) v \, \mathrm d f_t(x,v) \nonumber\\
&\quad+ \int \pi^\rho(t,x) \left[\,\int a^\rho(t,x,y)(w-v)\, \mathrm df_t(y,w)\right] \mathrm df_t(x,v),
\end{align}
by a straightforward computation, the last term is identically equal to zero thanks to \eqref{eq:detbal2}.
Hence
\begin{align}
\dfrac{{\mathrm d}}{\mathrm dt} \langle v(t)\rangle_{f}=\int \left(\partial_t \pi^\rho(t,x) + v \cdot \nabla_x \pi^\rho(t,x)\right) [v-\langle v(t) \rangle_{f}] \, \mathrm d f_t(x,v)
\end{align}
since \(\partial_t \int \pi^\rho(t,x) \mathrm d\rho_t(x)=0\).
By \eqref{in:boundpi}, it follows that
\begin{align}
\left| \dfrac{{\mathrm d}}{\mathrm dt} \langle v(t)\rangle_{f}\right| &\le \int \, \mathrm df_t(x,v)\pi^\rho(t,x)\left|\dfrac{\partial_t\pi^\rho(t)+v\cdot \nabla_x \pi^\rho(t)}{\pi^\rho(t)}\right|\int \pi^\rho(t,y)|v-w|\, \mathrm df_t(y,w)\nonumber\\
&\le \dfrac{2\text{Lip}(\varphi)}{m}[D^f_V(t)]^2.
\end{align}
By the asymptotic flocking \eqref{in:asymptflkin}, we conclude that
\[
\left| \dfrac{{\mathrm d}}{\mathrm dt} \langle v(t)\rangle_{f}\right|  \le  \dfrac{2 \text{Lip}(\varphi)}{m} [D^f_V(0)]^2 \mathrm{e}^{-2\eta t}.
\]
This implies that the limit in time of \(\langle v(t) \rangle_{f}\) exists and is given by 
\begin{equation}
\label{eq:vinff}
\mathbf{v}^f_\infty:= \langle v(0) \rangle_{f}+\int_0^\infty \int(\partial_t \pi^\rho(t)+v\cdot \nabla_x \pi^\rho(t)) v \, \mathrm df_t(x,v)\, {\mathrm d}t.
\end{equation}
Moreover, for any \(v \in \text{supp} f_t\)
\[
|v- \mathbf{v}_\infty^f| \le |v- \langle v(t) \rangle_{f}| + |\langle v(t) \rangle_{f}- \mathbf{v}_\infty^f|\le C{\mathrm e}^{-\eta t},
\]
with \(C:=D^f_V(0)+\text{Lip}(\varphi) [D^f_V(0)]^2/ \eta m\).
\end{proof}

\section{Uniform-in-time propagation of chaos: Proof of Theorem \ref{thm2}}
\label{sec4}
 We introduce the \(N\)-particle dynamics
\(
\{\bar X_i(t),\bar V_i(t)\}_{i=1}^N
\)
given by \(N\) independent copies of the system
\begin{equation}
\label{eq:interdyn}
\begin{cases}
\dot{\bar X}_i(t)=\bar V_i(t),\\
\dot{\bar V}_i(t)
=
F[f_t]\bigl(\bar X_i(t),\bar V_i(t)\bigr),
\end{cases}\qquad i \in \{1,\dots,N\}
\end{equation}
where \(f_t\) is a regular compactly supported solution of the kinetic Motsch--Tadmor model \eqref{eq:mesoVlasov} and \(F[f_t]\) is defined in \eqref{eq:meanforce}. 
We denote by
\begin{equation}
\delta X_i(t):=X_i^N(t)-\bar X_i(t),
\qquad
\delta V_i(t):=V_i^N(t)-\bar V_i(t),
\end{equation}
where \(\{X^N_i(t),V^N_i(t)\}\) and \(\{\bar X_i(t), \bar V_i(t)\}\) evolve with the same initial datum \((X_N, V_N) \in \mathrm{supp}(f_{\mathrm{in}}^{\otimes N})\).
Let $\widetilde F_{k,t}^{N}$ be the $k$-particle marginal of \((\mathcal{G}^N_t)_\# F^N_t\) and
let \(\widetilde f_t^{\otimes k}:=[(\mathcal{G}_t)_\#f_t]^{\otimes k}\).

By the definition of \(2\)-Wasserstein distance:
$$
\begin{aligned}
\mathcal W_2^2
\left(
\widetilde F_{k,t}^N,
\widetilde f_t^{\otimes k}
\right)
\leq
\sum_{i=1}^k
\int
\left|
\delta X_i(t)-t\Delta\mathbf  V_\infty
\right|^2
+
\left|
\delta V_i(t)-\Delta \mathbf V_\infty
\right|^2 \,\mathrm df^{\otimes N}_{\mathrm{in}}(X_N,V_N),
\end{aligned}
$$
where
$
\Delta \mathbf V_\infty
:=
\mathbf V_\infty^N-\mathbf v^f_\infty.
$
To bound the Wasserstein distance, it will be useful to introduce the following variance computed with respect to the distribution \(\{\pi_i^N(t)\}^N_{i=1}\) introduced in \eqref{def:eqmeas}:
\begin{equation}
\label{variance_X}
\mathcal{X}_N(t):= \sum_{i=1}^N \pi^N_i(t) \left|\delta X_i(t)-\langle \delta\boldsymbol{ X}(t)\rangle_{N}\right|^2= \frac{1}{2}\sum_{i,j=1}^N \pi^N_{i}(t) \pi^N_{j}(t) \left|\delta X_i(t)-\delta X_j(t)\right|^2
\end{equation}
and, similarly,
\[
\mathcal{V}_N(t):= \sum_{i=1}^N \pi^N_i(t) \left|\delta V_i(t)-\langle \delta \boldsymbol{V}(t)\rangle_{N}\right|^2= \frac{1}{2}\sum_{i,j=1}^N \pi^N_{i}(t) \pi^N_{j}(t) |\delta V_i(t)-\delta V_j(t)|^2,
\]
where the weighted averages \(\langle \delta \boldsymbol{X} \rangle_N\) and \(\langle \delta \boldsymbol{V}\rangle_N\) are defined in \eqref{eq:weightedN}.
Using that
\(
\pi^N_i(t)\geq\frac{\eta}{N},
\)
we obtain
\[
\frac1{N^2}
\sum_{i\neq j}^N
|\delta X_i(t)-\delta X_j(t)|^2
\leq
\frac{2}{\eta^2}\mathcal X_N(t),
\]
and analogously for \(\mathcal{V}_N(t)\).

\subsection{Step 1: Variance estimate}

We compute the time-derivative of the variances. We get
\begin{align}
\label{formula:var}
\dfrac{{\mathrm d}}{\mathrm dt} &\mathcal{X}_N(t)
=\sum_{i=1}^N \dot \pi_i^N(t) | \delta X_i(t) - \langle \delta \boldsymbol{X}(t)\rangle_N|^2+ 2 \sum_{i=1}^N  \pi_i^N(t)  [\delta X_i(t) - \langle \delta \boldsymbol{X}(t) \rangle_N]\cdot \dot{\delta X_i}(t),
\end{align}
and an analogous formula holds for \(\mathcal{V}_N(t)\).
By a direct computation, for \(i \in \{1,\dots,N\}\), it follows that
\begin{equation}
\label{formula:diffdyn}
\begin{cases}
    \dot{\delta X_i}(t)= {\delta V_i}(t)\\
    \dot{\delta V_i}(t)=\displaystyle{\sum_{j=1}^N
a_{ij}\left(\boldsymbol{X}^N(t)\right)
(\delta V_j(t)-\delta V_i(t))
+\mathcal{E}^X_i(t)
+
\mathcal{E}^{\text{c}}_i}(t)
\end{cases}
\end{equation}
where we have introduced two error quantities:
\begin{equation}
\label{error:x}
\mathcal{E}^X_i(t)
:=
\sum_{j=1}^N
\left[
a_{ij}\left(\boldsymbol{X}^N(t)\right)
-
a_{ij}\left(\bar{\boldsymbol{X}}(t)\right)
\right]
(\bar V_j(t)-\bar V_i(t)),
\end{equation}
and the \emph{consistency error}, defined by the discrepancy between the empirical force and the deterministic mean-field force evaluated along the independent trajectories \(\{\bar X_i(t), \bar V_i(t)\}_{i=1}^N\),
\begin{equation}
\label{error:cons}
\mathcal{E}_i^{\text{c}}(t)
:=
\sum_{j=1}^N
a_{ij}(\bar{\boldsymbol{X}}(t))
\bigl(\bar V_j(t)-\bar V_i(t)\bigr)
-
F[f_t]\bigl(\bar X_i(t),\bar V_i(t)\bigr).
\end{equation}
We apply the formula for the variance in \eqref{formula:var} and substitute \eqref{formula:diffdyn}, we get
\begin{align*}
\dfrac{{\mathrm d}}{\mathrm dt} &\mathcal{V}_N(t)
=\underbrace{\sum_{i=1}^N \dot \pi_i^N(t) | \delta V_i(t) - \langle \delta \boldsymbol{V}(t)\rangle_{N}|^2}_{A_1(t)} \\
&+\underbrace{2 \sum_{i,j=1}^N  \pi_i^N(t) a_{ij}(\boldsymbol{X}^N(t))[\delta V_j(t)- \delta V_i(t)]\cdot [\delta V_i(t) - \langle \delta \boldsymbol{V}(t) \rangle_{N}]}_{A_2(t)}\\
&+ \underbrace{2 \sum_{i=1}^N  \pi_i^N(t) \mathcal{E}^X_i(t)\cdot [\delta V_i(t) - \langle \delta \boldsymbol{V}(t) \rangle_{N}]}_{A_3(t)}
+\underbrace{2 \sum_{i=1}^N  \pi_i^N(t) \mathcal{E}_i^{\text{c}}(t)\cdot [\delta V_i(t) - \langle \delta \boldsymbol{V}(t) \rangle_{N}]}_{A_4(t)}.
\end{align*}
We study separately the four terms \(A_1(t),\dots,A_4(t)\).

\hspace{0.5 pt}

\noindent\underline{\emph{\(A_1(t)\) term}:} We multiply and divide by \(\pi^N_i(t)\), using \eqref{in:boundpi} in Lemma \ref{lemma:lemmapi} and \eqref{in:asymptfl}, we get
\begin{equation}
\label{est:A1}
A_1(t)\le \max_{i\in \{1,\dots,N\}} \left|\dfrac{\dot \pi^N_i(t)}{\pi^N_i(t)}\right|\mathcal{V}_N(t)\le \frac{2 \text{Lip}(\varphi)}{m} D^N_V(0) {\mathrm{e}^{-\eta t}}\mathcal{V}_N(t).
\end{equation}
\hspace{0.5 pt}

\noindent\underline{\emph{\(A_2(t)\) term}:}
We use the detailed-balance property \eqref{eq:detbal} of the distribution \(\{\pi^N_i(t)\}_{i=1}^N\). It holds that
\[
A_2(t)=- \sum_{i,j=1}^N \pi_i^N(t) a_{ij}(\boldsymbol{X}^N(t)) |\delta V_i(t)- \delta V_j(t)|^2 \le 0.
\]
Moreover, multiplying and diving by \(\pi_j^N(t)\) and using that \(\frac{a_{ij}}{\pi^N_j}\geq \eta^2\), we get
\[
 \sum_{i,j=1}^N \pi_i^N (t)a_{ij}(\boldsymbol{X}^N(t)) |\delta V_i(t)- \delta V_j(t)|^2 \ge \eta^2 \sum_{i,j=1}^N \pi_i^N(t) \pi_j^N(t)  |\delta V_i(t)- \delta V_j(t)|^2= 2\eta^2 \mathcal{V}_N(t).
\]
Hence
\begin{equation}
    \label{est:A2}
    A_2(t) \le - 2\eta^2 \mathcal{V}_N(t).
\end{equation}
\hspace{0.5 pt}

\noindent\underline{\emph{\(A_3(t)\) term}:}
By Cauchy--Schwarz, 
\[
A_3(t) \le 2\sqrt{\mathcal{V}_N(t)} \left(\sum_{i=1}^N{\pi_i^N(t)}|\mathcal{E}^X_i(t)|^2\right)^{\frac12}.
\]
We bound the second term involving \(\mathcal{E}_i^X(t)\). For any \(i \in \{1,\dots,N\}\), we get
\[
|\mathcal{E}_i^X(t)| \le \sum_{j=1}^N \left|a_{ij}\left(\boldsymbol{X}^N(t)\right)-a_{ij}\left(\bar{\boldsymbol{X}}(t)\right)\right| D_V^f(t),
\]
where we used that
\[
\max_{i,j \in \{1,\dots,N\}}|\bar V_i(t)- \bar V_j(t)|\leq D_V^f(t).
\]
By the triangle inequality, it is not difficult to show that
\[
\sum_{j=1}^N\left|a_{ij}\left(\boldsymbol{X}^N(t)\right)-a_{ij}\left(\bar{\boldsymbol{X}}(t)\right)\right|\le \text{Lip}(\varphi)\left(\frac 1m + \frac{M}{m^2}\right) \frac{1}{N} \sum_{j=1}^N |\delta X_i-\delta X_j|.
\]
Hence, by the Jensen's inequality,
\[
\left|\mathcal{E}^X_i(t)\right|^2 \le \left[D_V^f(t)\right]^2 \text{Lip}^2(\varphi)\left(\frac 1m + \frac{M}{m^2}\right)^2 \frac{1}{N} \sum_{j=1}^N |\delta X_i -\delta X_j|^2.
\]
Multiplying and dividing by \(\pi^N_j(t)\), by \eqref{variance_X} and using that \(\pi_j^N(t) N \geq \eta\), this implies that
\begin{equation}
\label{in:A3term}
\sum_{i=1}^N{\pi_i^N(t)}|\mathcal{E}^X_i(t)|^2 \le \left[D_V^f(t)\right]^2 \text{Lip}^2(\varphi)\left(\frac 1m + \frac{M}{m^2}\right) ^2 \frac{2}{\eta} \mathcal{X}_N(t).
\end{equation}
Using the exponential flocking property in \eqref{in:asymptflkin},
we conclude that
\begin{equation}
    \label{est:A3}
    A_3(t) \le \sqrt{\frac{8}{\eta}} D_V^f(0) \, \text{Lip}(\varphi)\left(\frac 1m + \frac{M}{m^2}\right) {\mathrm e}^{-\eta t}\sqrt{\mathcal{X}_N(t)\,\mathcal{V}_N(t)}.
\end{equation}

\hspace{0.5 pt}

\noindent\underline{\emph{\(A_4(t)\) term}: consistency error.}
This is the term where a law of large numbers is needed.
We use again the Cauchy--Schwarz inequality, to obtain
\begin{equation}
\label{est:A4}
    A_4(t)\le 
\sqrt{\mathcal{V}_N(t)} \left(\sum_{i=1}^N{\pi_i^N(t)}|\mathcal{E}_i^{\text{c}}(t)|^2\right)^{\frac12},
\end{equation}
and we keep this quantity in the following estimates.
At the end of the proof, we will need the following \(L^2\)-consistency estimate, which is standard in the literature:
\begin{equation}
\label{in:consest}
 \left \|\mathcal{E}_i^c(t) \right\|^2_{L^2( f_{\mathrm{in}}^{\otimes N})}=\int_{(\R^d \times \R^d)^{N}} [\mathcal{E}^{\text{c}}_i(t)]^2{\mathrm d}f^{\otimes N}_0(X_N, V_N) \le \frac{\left[2D_V^f(0)\right]^2}{\eta^2  N} \mathrm e^{-2\eta t}.
\end{equation}
Details can be found in the Appendix \ref{sec:appendix}.
\subsection{Step 2: Grönwall estimate}
Collecting the estimates in \eqref{est:A1}, \eqref{est:A2}, \eqref{est:A3} and \eqref{est:A4}, we get
\begin{align*}
\dfrac{{\mathrm d}}{\mathrm dt} &\mathcal{V}_N(t)
\le \frac{2 \text{Lip}(\varphi)}{m} D^N_V(0) {\mathrm{e}^{-\eta t}}\mathcal{V}_N(t)-2\eta^2 \mathcal{V}_N(t)\\
&+\sqrt{\frac{8}{\eta}} D_V^f(0) \, \text{Lip}(\varphi)\left(\frac 1m + \frac{M}{m^2}\right) {\mathrm e}^{-\eta t} \sqrt{\mathcal{X}_N(t)\,\mathcal{V}_N(t)}+\sqrt{\mathcal{V}_N(t)} \left(\sum_{i=1}^N{\pi_i^N(t)}|\mathcal{E}_i^{\text{c}}(t)|^2\right)^{\frac12}.
\end{align*}
Similarly, it is not difficult to show that
\begin{align*}
\dfrac{{\mathrm d}}{\mathrm dt}\mathcal{X}_N(t)
\le\frac{2 \text{Lip}(\varphi)}{m} D^N_V(0) {\mathrm{e}^{-\eta t}}\mathcal{X}_N(t)+2 \sqrt{\mathcal{X}_N(t)\, \mathcal{V}_N(t)}.
\end{align*}
The variance estimates established above imply that, almost surely and for almost every $t\geq0$,
\begin{equation}
\dfrac{\mathrm d}{\mathrm dt} \sqrt{\mathcal{X}_N(t)}
\leq
C_1 \mathrm e^{-\eta t}\sqrt{\mathcal{X}_N(t)}+\sqrt{\mathcal{V}_N(t)},
\end{equation}
and
\begin{equation}
\label{in:sqrtvn}
\dfrac{\mathrm d}{\mathrm dt}\sqrt{\mathcal{V}_N(t)}
\leq
-\eta^2 \sqrt{\mathcal{V}_N(t)}
+
\mathrm e^{-\eta t}\bigl(C_1\, \sqrt{\mathcal{V}_N(t)}+C_2\, \sqrt{\mathcal{X}_N(t)}\bigr)
+
\frac 12\left(\sum_{i=1}^N{\pi_i^N(t)}|\mathcal{E}_i^{\text{c}}(t)|^2\right)^{\frac12},
\end{equation}
where
\[
C_1:=\frac{\text{Lip}(\varphi)}{m} D^N_V(0), \quad C_2:=\sqrt{\frac{2}{\eta}} D_V^f(0) \, \text{Lip}(\varphi)\left(\frac 1m + \frac{M}{m^2}\right).
\]
Let
\(
\mathcal H_N(t):=\sqrt{\mathcal{X}_N(t)}+\frac{1}{\eta^2}\,\sqrt{\mathcal{V}_N(t)},
\)
we obtain
\begin{equation}
\dfrac{\mathrm d}{\mathrm dt}\mathcal H_N(t)
\leq
\left(\frac{C_1+C_2}{\eta^2}\right) \mathrm e^{-\eta t}\mathcal H_N(t)
+
\frac{\zeta_N(t)}{2\,\eta^2},
\end{equation}
where
\begin{equation}
\label{def:zeta}
\zeta_N(t):= \left(\sum_{i=1}^N{\pi_i^N(t)}|\mathcal{E}_i^{\text{c}}(t)|^2\right)^{\frac12}.
\end{equation}
Applying Grönwall's inequality, we obtain for any \(t \geq 0\),
\begin{equation}
\label{in:Gron}
\mathcal H_N(t)
\leq
\frac{1}{2\eta^2}\int_0^t
\exp\left(
\left[\frac{C_1+C_2}{\eta^2}\right]\int_s^t \mathrm e^{-\eta \tau}\,{\mathrm d}\tau
\right)
\zeta_N(s)\,{\mathrm d}s
\leq
\frac{{\mathrm e}^{\frac{C_1+C_2}{\eta^3}}}{2\, \eta^2}\int_0^t\zeta_N(s)\,\mathrm ds.
\end{equation}
By the previous estimates, by taking the \(L^2(f_{\mathrm{in}}^{\otimes N})\)-norm and using the consistency estimate \eqref{in:consest}, we get
\[
\|\mathcal{H}_N(t)\|_{L^2}\leq \frac{C_3}{\sqrt N},
\qquad
C_3:=\frac{ D_V^f(0){\mathrm e}^{\frac{C_1+C_2}{\eta^3}}}{ \eta^{4}}.
\]
Hence, applying the variation-of-constants formula in \eqref{in:sqrtvn},
we obtain
$$
\sqrt{\mathcal V_N(t)}
\leq
\int_0^t
\mathrm e^{-\eta^2(t-s)}
\left(
\mathrm e^{-\eta s}C_4\,{\mathcal H_N(s)}+\frac12 \zeta_N(s)
\right)\,\mathrm ds,
$$
where
$
C_4:=C_1 \eta^2+C_2.
$
Thus, taking the \(L^2(f_{\mathrm{in}}^{\otimes N})\)-norm yields
\begin{equation}
\label{in:mainv}
\left\|\sqrt{\mathcal{V}_N(t)}\right \|_{L^2}
\leq C_5\frac{e^{-\eta^2 t}\,t}{\sqrt{N}},
\qquad t\geq0,
\end{equation}
where 
$$
C_5:= \left(C_3 \,C_4 + \frac{D_V^f(0)}{\eta} \right).
$$
\subsection{Step 3: Velocity discrepancies}

We now study the convergence of the velocity discrepancies \(\{\delta V_i(t)\} \). 
Differentiating and using the dynamics of \(\delta V_i\) in \eqref{formula:diffdyn}, we obtain
$$
\begin{aligned}
\frac{\mathrm d}{\mathrm dt}\langle\delta \boldsymbol{V}(t)\rangle_{N}
&=
\sum_{i=1}^N\dot\pi_i^N(t)\delta V_i(t)
+
\sum_{i,j=1}^N
\pi_i^N(t)a_{ij}(\boldsymbol{X}^N(t))
(\delta V_j(t)-\delta V_i(t))
\\
&\qquad
+
\sum_{i=1}^N\pi_i^N(t)\mathcal E_i^X(t)
+
\sum_{i=1}^N\pi_i^N(t)\mathcal E_i^{\rm c}(t),
\end{aligned}
$$
where \(\mathcal{E}_i^X(t)\) and \(\mathcal E_i^{\rm c}(t)\) are introduced in \eqref{error:x} and \eqref{error:cons}. 

By the detailed balance relation \eqref{eq:detbal}, the second term vanishes. Moreover, since
\(\sum_i\dot\pi_i^N(t)=0\),
$$
\sum_i\dot\pi_i^N(t)\delta V_i(t)
=
\sum_i\dot\pi_i^N(t)
\bigl(\delta V_i(t)-\langle\delta \boldsymbol{V}(t)\rangle_{N}\bigr).
$$
Therefore, multiplying and dividing the first term by \(\pi_i^N(t)\), by Cauchy--Schwarz and using that \(\sum_i \pi_i^N=1\),
$$
\begin{aligned}
\left|\frac{\mathrm d}{\mathrm dt} \langle\delta \boldsymbol{V}(t)\rangle_{N} \right|
\leq
\max_i
\left|
\frac{\dot\pi_i^N(t)}{\pi_i^N(t)}
\right|
\sqrt{\mathcal V_N(t)}
+
\left(
\sum_i\pi_i^N(t)|\mathcal E_i^X(t)|^2
\right)^{1/2}
+
\zeta_N(t),
\end{aligned}
$$
where \(\zeta_N(t)\) is defined in \eqref{def:zeta}.
Using the bounds in \eqref{in:boundpi} and \eqref{in:A3term},
we conclude that
$$
\left|\frac{\mathrm d}{\mathrm dt} \langle\delta \boldsymbol{V}(t)\rangle_{N} \right|
\leq
C_6\,\mathrm e^{-\eta t}
\mathcal H_N(t)+\zeta_N(t),
$$
with 
$$
C_6:=\frac{\mathrm{Lip}(\varphi)}{m}\left[2\, D_V^N(0)\, \eta^2 + \left(1+\frac{1}{\eta}\right)\sqrt{\frac{2}{\eta}}D_V^f(0)\right].
$$
Hence, taking the \(L^2(f_{\mathrm{in}}^{\otimes N})\)-norm,
\begin{equation}
\label{eq:qdot}
\left\|\frac{\mathrm d}{\mathrm dt} \langle\delta \boldsymbol{V}(t)\rangle_{N}\right\|_{L^2}
\leq
\frac{C_7}{\sqrt N}\mathrm e^{-\eta t},
\end{equation}
with  $C_7:=C_6+2\, D_V^f(0)/\eta$. 
Since
$
V_i^N(t)\to \mathbf V_\infty^N
$
and 
$
\bar V_i(t)\to \mathbf v^f_\infty,
$
we have
$
\langle \delta V(t) \rangle_{N}\to \Delta \mathbf V_\infty
$ as $t\to\infty$.
Thus, from \eqref{eq:qdot},
\begin{equation}
\label{eq:qtail}
\|\langle \delta V(t)\rangle_{N}-\Delta \mathbf V_\infty\|_{L^2}
\leq
\frac{C_7}{\eta\, \sqrt N}\mathrm e^{-\eta t}.
\end{equation}
We can now study the convergence of each velocity discrepancy. Indeed,
\begin{equation}
\label{eq:triangle}
\delta V_i(t)-\Delta \mathbf V_\infty
=
\bigl(\delta V_i(t)-\langle \delta V(t)\rangle_{N}\bigr)
+
(\langle \delta V(t)\rangle_{N}-\Delta \mathbf V_\infty\bigr).
\end{equation}
By exchangeability and the lower bound
\(\pi_i^N(t)\geq\eta/N\),
$$
\begin{aligned}
\|\delta V_i(t)-\langle \delta V(t)\rangle_{N}\|^2_{L^2}
=
\frac1N
\sum_{j=1}^N
\|
\delta V_j(t)-\langle \delta V(t)\rangle_{N}\|_{L^2}^2
\leq
\frac1\eta
\,\left\|\sqrt{\mathcal V_N(t)}\right\|_{L^2}^2.
\end{aligned}
$$
Combining this inequality with \eqref{in:mainv} and \eqref{eq:qtail}, 
we obtain
\begin{equation}
\label{eq:velocity-Galilean}
\left\|
\delta V_i(t)-\Delta \mathbf V_\infty
\right\|_{L^2}
\leq
\frac{C_8}{\sqrt N}\left(\mathrm e^{-\eta^2 t}t+\mathrm e^{-\eta t}\right)
\end{equation}
with 
$$
C_8:=\frac{C_5}{\sqrt{\eta}}+\frac{C_7}{\eta}.
$$
\subsection{Step 4: Conclusion}
Since the two systems start from the same initial configuration,
\(\delta X_i(0)=0\), and therefore
$
\delta X_i(t)
=
\int_0^t\delta V_i(s)\,\mathrm ds.
$
It follows that
$$
\delta X_i(t)-t\Delta \mathbf V_\infty
=
\int_0^t
\left(
\delta V_i(s)-\Delta \mathbf V_\infty
\right)\,\mathrm ds.
$$
By Minkowski's inequality and \eqref{eq:velocity-Galilean},
\begin{equation}
\label{eq:position-Galilean}
\begin{aligned}
\left\|
\delta X_i(t)-t\Delta \mathbf V_\infty
\right\|_{L^2}
\leq
\int_0^t
\left\|
\delta V_i(s)-\Delta \mathbf V_\infty
\right\|_{L^2} \mathrm ds
\leq
\frac{C_8}{\sqrt N}\left( \frac{1}{\eta^4}+\frac 1 \eta\right).
\end{aligned}
\end{equation}
Therefore, by
$$
\begin{aligned}
\mathcal W_2^2
\left(
\widetilde F_{k,t}^N,
\widetilde f_t^{\otimes k}
\right)
\leq
\sum_{i=1}^k \int
\left|
\delta X_i(t)-t\Delta \mathbf V_\infty
\right|^2
+
\left|
\delta V_i(t)-\Delta \mathbf V_\infty
\right|^2\, \mathrm{d} f_{\mathrm{in}}^{\otimes N}(X_N,V_N),
\end{aligned}
$$
and using \eqref{eq:velocity-Galilean} and
\eqref{eq:position-Galilean}, we conclude that
$$
\sup_{t\geq0}
\mathcal W_2^2
\left(
\widetilde F_{k,t}^N,
\widetilde f_t^{\otimes k}
\right)
\leq
\frac{k}{N}\, C_8^2 \left[\left( \frac{1}{\eta^4}+\frac 1 \eta\right)^2+ \left(1+\frac{1}{\eta^2}\right)^2 \right],
\qquad
1\leq k\leq N.
$$
The conclusion follows by taking
\begin{equation}
    \label{eq:finalC}
C:= C_8 \sqrt{\left( \frac{1}{\eta^4}+\frac 1 \eta\right)^2+ \left(1+\frac{1}{\eta^2}\right)^2},
\end{equation}
and by noticing that, in all the introduced constants, \(D_V^N(0) \le D_V^f(0)\).
\qed

\addtocontents{toc}{\protect\setcounter{tocdepth}{0}}

\section*{Acknowledgments}
\noindent TK and LT acknowledge the support of the OeAD (project WTZ FR 04/2026). SR acknowledges the European Union through the ERC CoG KiLiM grant, project number 101125162\footnote{Views and opinions expressed are however those of the authors
only and do not necessarily reflect those of the European Union or the European Research
Council. Neither the European Union nor the granting authority can be held responsible for
them.}.

\subsection*{Use of AI technologies}
During the preparation of this manuscript, ChatGPT-5.6 Sol (OpenAI) was used for preliminary brainstorming and conceptual exploration,
as well as to improve the clarity, readability, and presentation of the manuscript.
All mathematical arguments were examined, rigorously developed, and independently verified by the authors, who take full responsibility for the mathematical correctness and final content of the manuscript.

\addtocontents{toc}{\protect\setcounter{tocdepth}{1}}

\appendix
\section{Estimate of the consistency error}
\label{sec:appendix}

The aim is to show the \(L^2\)-consistency estimate~\eqref{in:consest}. We consider the \(L^2\)-norm
\begin{align*}
    \left\|\mathcal E_i^c(t) \right\|^2_{L^2( f_{\mathrm{in}}^{\otimes N})}=\int_{(\R^d \times \R^d)^{N}} [\mathcal{E}^{\text{c}}_i(t)]^2{\mathrm d}f^{\otimes N}_{\mathrm{in}}(X_N,V_N),
\end{align*}
where we recall that
\[
\mathcal{E}_i^c(t):=\sum_{j=1}^N a_{ij}(\bar{\boldsymbol{X}}(t)) \bigl(\bar V_j(t)-\bar V_i(t)\bigr)- F[f_t]\bigl(\bar X_i(t),\bar V_i(t)).
\]
By triangle inequality, we have
\[
 \left\|\mathcal E_i^c(t) \right\|_{L^2( f_{\mathrm{in}}^{\otimes N})}\le  \left\|\mathcal E_i^{c,1}(t) \right\|_{L^2( f_{\mathrm{in}}^{\otimes N})}+ \left\|\mathcal E_i^{c,2}(t) \right\|_{L^2( f_{\mathrm{in}}^{\otimes N})},
\]
where
\[
\mathcal{E}_i^{c,1}(t):=\frac 1N\sum_{j=1}^N \left[\dfrac{\varphi(|\bar X_i(t)-\bar X_j(t)|}{\frac 1N\sum_h \varphi(|\bar X_i(t)-\bar X_h(t)|)} - \dfrac{\varphi(|\bar X_i(t)-\bar X_j(t)|}{\int \varphi(|\bar X_i(t)-y|)\mathrm d\rho_t(y)}\right] \bigl(\bar V_j(t)-\bar V_i(t)\bigr),
\]
and
\[
\mathcal{E}_i^{c,2}(t):=\frac 1N\sum_{j=1}^N\dfrac{\varphi(|\bar X_i(t)-\bar X_j(t)|}{\int \varphi(|\bar X_i(t)-y|)\mathrm d\rho_t(y)} \bigl(\bar V_j(t)-\bar V_i(t)\bigr)- F[f_t](\bar X_i(t), \bar V_i(t)).
\]
By \eqref{in:lowupbound} and using that \(|\bar V_i(t) -\bar V_j(t)| \le D_V^f(t)\), we first obtain
\[
\begin{aligned}
|\mathcal E_i^{c,1}(t)|
&\le
\frac{MD_V^f(t)}{m^2}
\left|
\frac1N\sum_{h=1}^N
\varphi(|\bar X_i(t)-\bar X_h(t)|)
-
\int \varphi(|\bar X_i(t)-y|)\,\mathrm d\rho_t(y)
\right|.
\end{aligned}
\]
Using the independence of the particles,
\[
\left\|
\frac1N\sum_{h=1}^N
\varphi(|x-\bar X_h|)
-
\int \varphi(|x-y|)\,\mathrm d\rho_t(y)
\right\|_{L^2(f_{\mathrm{in}}^{\otimes N})}
\le \frac{M}{\sqrt N}.
\]
Hence
\[
\
\|\mathcal E_i^{c,1}(t)\|_{L^2(f_{\mathrm{in}}^{\otimes N})}
\le
\frac{D_V^f(t)}{\eta^2\sqrt N}.
\]
Similarly, concerning \(\mathcal{E}_i^{\mathrm c,2}(t)\), the terms in
\[
\frac1N\sum_{j=1}^N
\frac{\varphi(|x-\bar X_j|)}
{\int \varphi(|x-y|)\,\mathrm d\rho_t(y)}
(\bar V_j-v)
\]
are independent, with mean
\(
F[f_t](x,v).
\)
Moreover,
\[
\left|
\frac{\varphi(|\bar X_i-\bar X_j|)}
{\int \varphi(|\bar X_i-y|)\,\mathrm d\rho_t(y)}
(\bar V_j-\bar V_i)
\right|
\le
\frac{D_V^f(t)}{\eta}.
\]
Therefore,
\[
\|\mathcal E_i^{c,2}(t)\|_{L^2(f_{\mathrm{in}}^{\otimes N})}
\le
\frac{D_V^f(t)}{\eta\sqrt N}.
\]
Using the exponential bound in \eqref{in:asymptflkin},
we finally get the thesis.

\bibliographystyle{abbrv}
\bibliography{biblio}
 
\end{document}